\documentclass{amsart}
\usepackage{amsmath,amsfonts,amssymb,amsthm, esint}
\usepackage[english]{babel}
\usepackage{enumerate}
\usepackage{dsfont}
\usepackage{comment}

\usepackage[colorlinks,citecolor=blue]{hyperref}
\usepackage{epsfig,graphicx} 

\newtheorem{theorem}{Theorem}[section]

\newtheorem{lemma}[theorem]{Lemma}
\newtheorem{proposition}[theorem]{Proposition}

\theoremstyle{remark}
\newtheorem{remark}[theorem]{Remark}
\newtheorem{example}[theorem]{Example}

\numberwithin{equation}{section}
\numberwithin{figure}{section}

\newcommand{\R}{\mathbb{R}}
\newcommand{\N}{\mathbb{N}}

\def\Reach{\mathrm{Reach}\,}

\def\gammalim{\operatornamewithlimits{\Gamma-lim}}
\def\gammalimsup{\operatornamewithlimits{\Gamma-lim\,sup}}
\def\gammaliminf{\operatornamewithlimits{\Gamma-lim\,inf}}

\begin{document}
	\author{Nikita Puchkin}
	\address{HSE University
	\and Institute for Information Transmission Problems, Moscow, Russian Federation
	}
	\email{npuchkin@hse.ru}

	\author{Vladimir Spokoiny}
	\address{
	Weierstrass Institute of Applied Analysis and Stochastics
	\and
	Humboldt University, Berlin, Germany
	\and HSE University
	\and Institute for Information Transmission Problems, Moscow, Russian Federation
	}
	\email{spokoiny@wias-berlin.de}

	\author{Eugene Stepanov}
	\address{St.Petersburg Branch of the Steklov Mathematical Institute of the Russian Academy of Sciences,
	Fontanka 27, 191023 St.Petersburg, Russian Federation
	\and
	Scuola Normale Superiore, Pisa, Italy
	\and
	HSE University, Moscow, Russian Federation		
	}
	\email{stepanov.eugene@gmail.com}
	
	\author{Dario Trevisan}
	\address{Dario Trevisan, Dipartimento di Matematica, Universit\`a di Pisa \\
	Largo Bruno Pontecorvo 5 \\ I-56127, Pisa}
	\email{dario.trevisan@unipi.it}
	
	\thanks{The publication was supported by the grant for research centers in the field of AI provided by the Analytical Center for the Government of the Russian Federation (ACRF) in accordance with the agreement on the provision of subsidies (identifier of the agreement 000000D730321P5Q0002) and the agreement with HSE University \textnumero70-2021-00139.}
	\date{\today}
	
	\title[Reconstruction of manifold embeddings via intrinsic distances]{Reconstruction of manifold embeddings into Euclidean spaces via intrinsic distances}
	
	\begin{abstract}
		We consider the problem of reconstructing  
		an embedding of a compact connected Riemannian manifold in a Euclidean space up to an almost isometry,
		given the information
		on intrinsic distances between points from its ``sufficiently large'' subset. This is one of the classical manifold learning problems.
		It happens that the most popular
		methods to deal with such a problem, with a long history in data science, namely, the classical Multidimensional scaling (MDS) and the Maximum variance unfolding (MVU), actually miss the point and may provide results very far from an isometry; moreover, they may even give no bi-Lipshitz embedding. We will provide an easy variational formulation of this problem, which leads to an algorithm always providing an almost isometric embedding with the distortion of original distances as small as desired (the parameter regulating the upper bound for the desired distortion is an input parameter of this algorithm).   
	\end{abstract}
	 
	\maketitle
		
\section{Introduction}
	Let $M$ be a smooth, connected, compact Riemannian manifold endowed with its intrinsic (geodesic) 
	distance $d_M$. Further, we will further consider $M$ to be embedded in some Euclidean space $\R^n$.
	Assume that we are given a sample $\{d_{ij}\}$ of pairwise distances between points of some point cloud
	$\{y_i\}\subset M$, i.e. $d_{ij}:= d_M(y_i, y_j)$. Our goal is to reconstruct an almost isometric embedding of $M$,
	or just of its subset  $\{y_i\}$, into  $\R^n$ based on the observed sample.
	In other words, we are interested in an algorithm, which, based on the input $\{d_{ij} \}$, produces a set $\{x_i\} \subset \Sigma$, with $\Sigma \subset \R^n$ being some other embedded manifold endowed with its intrinsic distance $d_\Sigma$, so that $d_\Sigma(x_i, x_j) \approx d_{ij}$,
	where the approximate inequality means that the distortion does not exceed a desired level.
	Note that in data science applications the set $\{y_i\}\subset M$ is of course finite, though its cardinality $N\in \N$ is usually quite large.
	
\subsection*{Existing results and methods}
There is a vast literature both in statistics and computational geometry 
on manifold reconstruction. The majority of existing methods are based directly on the finite point cloud $\{y_i\}_{i = 1}^N$, which is assumed to be known up to some errors.
 In particular, in~\cite{gppvw12, bg14, mms16, al18, fikln18, filn19, tsay19, ps22}, the authors consider the problem of $C^2$ manifold reconstruction based on a finite sample possibly corrupted with small zero-mean additive noise. In applications, one usually employs the respective methods to reduce the dimensionality of the known high-dimensional data. Note that this setup is much simpler than the one we consider, where, instead of the point cloud itself, we have only information on the respective distance matrix. Access to the point cloud allows to construct estimates of projectors onto tangent spaces to the manifold and then use them to reconstruct the manifold itself. For instance,~\cite{bg14} and~\cite{al18} used tangential Delaunay complexes. The approach of~\cite{mms16, tsay19} relied on local PCA estimates. In~\cite{ps22}, one iteratively uses a PCA-like procedure to successively improve projector estimates. In~\cite{fikln18, filn19}, the so-called putative manifold is used, that is a set of points solving a nonlinear system of equations.
Another class of manifold reconstruction methods from a point cloud is based on  random projections similarly to the classical Johnson-Lindenstrauss lemma (see, e.g.,~\cite{baraniuk07,hegde07,clarkson08}). 
Finally, in~\cite{al19, sober20} one studies the case when $M$ is a $C^m$ submanifold of $\R^n$: the authors used minimizers of a weighed sum of square errors to estimate not only the projectors but also higher order tensors up to order $m$. As a result, the guarantees on the Hausdorff distance between $M$ and the constructed estimate are much stronger in this case than those for the methods using only the first-order expansions.

The problem we are considering, when only pairwise distances $\{d_{ij} \}_{i, j = 1}^N$ are given, is somewhat less studied. It is worth mentioning, though, that the two problems of manifold reconstruction, the one directly from the point cloud and the other from just a distance matrix, are inherently related. In fact, many methods to solve the former actually contain, as a core part, some method to solve the latter.  As an example, the Isomap manifold embedding  algorithm~\cite{tsl00}, often used in applications for the purpose of data dimension reduction, contains as a core the classical multidimensional scaling (MDS) algorithm that deals only with distance matrices. 

Note that we are interested in reconstructing the embedding of the original manifold $M$ into an Euclidean space (e.g., for the purpose of data visualization), as opposed to the problem of reconstructing an abstract manifold $M$ (e.g., determined by its metric tensor). The latter is solved in~\cite{fefferman-Ivanov-2019reconstrmanif}, but its solution does not provide any explicit finite-dimensional embedding. Of course, once the metric tensor is reconstructed, one might also reconstruct an embedding by, say, some computational version of the Nash embedding theorem, but such a double-step procedure is unreasonably complicated.
Therefore, it is prompting to search for a direct algorithm to solve the posed problem. Such algorithms have already been proposed and are quite widely used in applications.
However, we will show that two basic and widely used algorithms, \textit{multidimensional scaling} (MDS) and \textit{maximum variance unfolding} (MVU), may infinitely distort the original distances even in simple situations. 
As a consequence, the methods relying on the classical MDS (e.g., Isomap \cite{tsl00}) may inherit such an undesired property. 
Note that some newer heuristic methods of dimension reduction 
with steadily growing popularity, like SNE or t-SNE~\cite{HintonRoweis-sne02,MaatenHinton-tsne08}, also perform 
reconstruction of data points just from the distances. Unfortunately, they do not have any rigorous guarantees on the distortion of pairwise distances (and, anyhow, it is clear one might expect at most some bounds on distance distortion ``in average'', but not uniform).
An attempt to understand t-SNE was made in a recent work~\cite{arora2018tsne-analysis}, but the authors only managed to show that t-SNE is able to keep the cluster structure in the data. Finally, the method proposed in~\cite{boissonnat17} also requires just a distance matrix as an input and reconstructs a manifold homeomorphic to the original one, but there are no upper bounds on the distance distortion for this method or its modifications in the literature.

\subsection*{Our contribution} On the contrary, in the present paper, we suggest a quite simple direct algorithm performing a manifold embedding in polynomial time and provide non-asymptotic upper bounds on the relative distortion of pairwise distances that are as small as desired (the requirement for the smallness of distortion is itself an input datum). It is also worth mentioning that, in fact, the algorithm we provide works not only with smooth Riemannian manifolds but rather with a far more general class of compact subsets of a Euclidean space connected by rectifiable arcs and satisfying some curvature estimate (e.g., having a positive reach); this estimate (or the lower bound for the reach) and the intrinsic diameter of the set have to be a priori known as they are also input parameters of the algorithm. 

\subsection*{Plan of the paper} The rest of the paper is organized as follows. 
Section~\ref{sec_MDSMVU} is dedicated to the analysis of MDS and MVU. In particular, relying on the result from~\cite{adams-blumstein-kassab2020MDS}, we show that MDS, applied to a unit circumference, produces a snowflake-like closed curve which is just H\"{o}lder continuous, and hence,  infinitely distorts the original distances. Our algorithm will be provided in Section~\ref{ssec_SDP1}. It is based on a semidefinite programming problem, and, consequently, runs in a polynomial time. The analysis of our approach is based on a variational setting proposed in Section~\ref{sec_varDistreconstr} and on a simple $\Gamma$-convergence result  (Theorem~\ref{th_Vdistest1}). As an application in Section~\ref{sec_Cech}, we show that this method can be used also for topological data reconstruction, i.e.\ for computing \v{C}ech cohomologies, and provide explicit estimates on the input parameters for this purpose. Finally, in Section~\ref{sec_numericsDistrreconstr} we provide some numerical experiments to illustrate the performance of the proposed algorithm.

\section{Notation and preliminaries}

For a metric space $E$ equipped with distance $d$ and a curve $\theta\colon [0,1]\to E$ 
we denote by $|\dot\theta|$ its \emph{metric derivative} and by
\[
	\ell(\theta):=\int_0^1 |\dot\theta|(t)\, dt 
\]
its \emph{parametric length}. The notation $B_r(x)\subset E$ stands for the open ball of radius $r>0$ with center $x\in E$.
The Euclidean norm is denoted by $|\cdot|$.

For a set $M\subset\R^n$ and $\varepsilon>0$ let $(M)_\varepsilon\subset\R^n$ to be its open $\varepsilon$-neighborhood, i.e.\
$(M)_\varepsilon:=\cup_{x\in N} B_\varepsilon(x)$. We recall the notion of \textit{reach} of $M$ introduced by Federer in~\cite{federer1959curvmeas}
and defined by
\[
\Reach(M):=\sup \left\{\varepsilon>0\colon \mbox{every point of $(M)_\varepsilon$ has a unique projection on $M$} \right\}.
\]

We further assume that the function space $C(M;\R^n)$ of continuous functions on $M$ with values in $\R^n$ is equipped with the usual unifom norm.
For a set $S\subset C(M;\R^n)$ we denote
\[
	\chi_{S} (f) :=
	\left\lbrace
	\begin{array}{rl}
		0, & f\in S,\\
		+\infty, & f\not\in S. 
	\end{array}
	\right.   
\] 
For an $m\times n$ matrix $X$ we denote as usual by $X^T$ its transpose. Vectors are silently identified with columns. 
The notation  $\mathrm{diag}\, (\alpha_1,\ldots, \alpha_r)$ stands for the diagonal matrix with entries
$(\alpha_1,\ldots, \alpha_r)$ over the diagonal. By $x\cdot y$ we denote the usual scalar product of vectors $x\in \R^n$ and $y\in \R^n$.
For any real numbers $a$ and $b$ the notation $a \wedge b$ stands for $\min\{a, b\}$.

For the general theory of $\Gamma$-convergence we refer the reader to~\cite{bra06-Gammaconv}, 
wherefrom we borrow also the respective notation.

\section{Main existing methods}\label{sec_MDSMVU}

The existing algorithms in manifold learning aimed at manifold reconstruction from intrinsic distances are quite numerous, but many of them are very closely related to 
 just two basic ones, \textit{multidimensional scaling} (MDS) and \textit{maximum variance unfolding} (MVU), which are aimed at reconstructing the locations of the points $y_j$ up to an isometric (or almost isometric) embedding.
This would be the case if the algorithm with an input $\{d_{ij}\}_{i,j=1}^N$, produced a set of points $\{x_i^N\}_{i=1}^N\subset \R^n$ such that the functions $f_N$, defined as $f_N(y_i):=x_i^N$, tend to some $f\colon M\to \R^n$ with $\Sigma:=f(M)$ (almost) isometric to $M$, when $N\to\infty$. 
Unfortunately, as we show below, the existing methods in general miss this point.

\subsection{Multidimensional scaling (MDS)}
The classical \textit{multidimensional scaling} (MDS) introduced by Torgerson and further developed by many authors (see chapter~6 of~\cite{wang2012geometric} and references therein),
has been formulated for the situation when the distance $d_M$ is Euclidean (which happens, e.g., when $M$ is a convex subset of $\R^m$).
In practice however MDS method is quite frequently applied when the distance $d_M$ is not necessarily Euclidean.
This however in general does not allow to reconstruct the embedding of the original manifold $M$ 
up to an (almost) isometry, as the following example shows.
  
\begin{example}\label{ex_LKassabS1}
	We follow the calculations from~\cite{adams-blumstein-kassab2020MDS} of the MDS embedding of the finite uniform samples of the unit circumference $M$.
Namely,  
if $\{y_i\}_{i=1}^N$ is a set of equally spaced points in $M$, then
proposition~7.2.6 of~\cite{adams-blumstein-kassab2020MDS} shows that
the MDS embedding of $\{y_i\}_{i=1}^N$
in $\R^n$ lies, up to a rigid motion, on the closed curve 
$\gamma_N\colon [0,2\pi]\to \R^n$ defined by
\begin{align*}
\gamma_N &(t) := \\
& (a_1^N \cos(t), a_1^N \sin(t), \ldots, a_{2k+1}^N  \cos((2k+1)t), a_{2k+1}^N \sin((2k+1)t), \ldots) \in\R^n,
\end{align*}
where $\lim_N a_j^N =a_j:=\sqrt{2}/j$ (with $j$ odd). Clearly, in the limit $N\to\infty$ and $n\to \infty$ this gives a
closed snowflake-like curve $\gamma$ homeomorphic but not isometric (nor even bilipschitz) to $M$; in fact, 
\begin{equation}\label{eq_snowflakeMDS1}
|\gamma(t)-\gamma(s)|=2\sqrt{\pi} |t-s|^{1/2}.
\end{equation}
To prove~\eqref{eq_snowflakeMDS1}, we calculate
\begin{equation}\label{eq_snowflakeMDS2}
\begin{aligned}
|\gamma(t)-\gamma(s)|^2 &= \sum_{k=0}^\infty\frac{4}{(2k+1)^2} \left(\sin (2k+1)t-\sin (2k+1)s\right)^2 \\
& \qquad + \sum_{k=0}^\infty\frac{4}{(2k+1)^2} \left(\cos (2k+1)t-\cos (2k+1)s\right)^2  \\
&=16\sum_{k=0}^\infty \frac{1}{(2k+1)^2} -  16\sum_{k=0}^\infty \frac{\cos\left((2k+1)(t-s)\right)  }{(2k+1)^2}\\
&= 2\pi^2 - 16 \sum_{k=0}^\infty \frac{\cos\left((2k+1)(t-s)\right)}{(2k+1)^2}.
\end{aligned}
\end{equation}
But from the Fourier expansion
\begin{align*}
|x|= \frac{\pi}{2} - \frac{4}{\pi} \sum_{k=0}^\infty \frac{\cos (2k+1)x}{(2k+1)^2}
\end{align*}
for $x\in [-\pi,\pi]$, we get
\begin{align*}
\sum_{k=0}^\infty \frac{\cos(2k+1)(t-s)}{(2k+1)^2} = \frac{\pi^2}{8}-\frac{\pi|t-s|}{4}, 
\end{align*}
which plugged in~\eqref{eq_snowflakeMDS2} gives~\eqref{eq_snowflakeMDS1} as claimed.
\end{example}

Deeper results on what is in fact reconstructed by applying the MDS to a generic metric measure space, together with further examples like MDS on a multidimensional sphere and on a flat torus, can be found in~\cite{KroshSteTrev22-mds} and also in~\cite{LimMemoli2022-mds}.

\subsection{Maximum variance unfolding (MVU)}

The method of \textit{maximum variance unfolding} (MVU) (alternatively called also \textit{semidefinite embedding} (SDE)), has been introduced by Weinberger
and Saul, see chapter~9.1 of~\cite{wang2012geometric}, and amounts to finding the points $x_i^N\in \R^n$, $i=1,\ldots, N$
given the distance matrix $\{d_{ij}\}_{i,j=1}^N$, by maximizing the total variance functional
\[
\mathrm{var}(x_{1},x_{2},...,x_{N}):=\sum _{i,j=1}^N |x_{i}-x_{j}|^{2}
\]
subject to the set of constraints
\begin{equation}\label{eq_MVUconstr1}
|x_{i}-x_{j}|^{2} = d_{ij}^2 \quad\mbox{whenever $y_i$ is close to $y_j$}.
\end{equation}
The condition of ``$y_i$ close to $y_j$'' is understood differently in different versions of MVU, but most commonly as
$d_{ij}\leq \varepsilon$ for some fixed $\varepsilon>0$, so that~\eqref{eq_MVUconstr1} becomes
\begin{equation}\label{eq_MVUconstr2}
|x_{i}-x_{j}|^{2} = d_{ij}^2 \quad\mbox{whenever $d_{ij}\leq \varepsilon$}.
\end{equation}
The constraints~\eqref{eq_MVUconstr1} (or in particular~\eqref{eq_MVUconstr2}) are reformulated in an equivalent way in terms of
the Gram matrix $K$ with entries $K_{ij}:= x_i\cdot x_j$ so that
the above maximization becomes \textit{semidefinite programming} problem.

Similarly to Example~\ref{ex_LKassabS1} it is easy to show that MVU, and even more, any method trying to preserve locally the distances as Euclidean ones, in general not only does not allow to reconstruct the embedding of the original manifold $M$ up to an (almost) isometry, but even worse, the constraints~\eqref{eq_MVUconstr2} may not allow to reconstruct 
even something vaguely similar to $M$, as the example below shows.

\begin{example}\label{ex_MVUS1}
	Taking again $M$ to be a unit circumference $S^1$, and $\{y_i\}_{i=1}^N$ to be a set of equally spaced points in $M$,
	suppose that $\{x_i^N\}_{i=1}^N\subset\R^n$ satisfy~\eqref{eq_MVUconstr2}, i.e\ 
\[
|x_{i}^N-x_{j}^N|^{2} = d_{ij}^2 \quad\mbox{whenever $d_{ij}\leq \varepsilon$}
\]	
for some fixed $\varepsilon>0$, and that continuous functions
$f_N\colon M\to\R^n$ satisfying $f_N(y_j)=x_j^N$ for all $j=1,\ldots, N$ 
converge as $N\to \infty$ to some continuous function
$f\colon M\to\R^n$.
Then one has
\[
|f(u)-f(v)|= d_M(u,v) \quad\mbox{whenever $d_M(u,v) < \varepsilon$}.
\]
Parameterizing $M$ in a natural way over $[0,2\pi]$ by a curve $\theta\colon [0,2\pi]\to M$, $\theta(t):=(\cos t,\sin t)$,
for the curve $\gamma\colon [0,2\pi]\to \R^n$ defined by $\gamma(t):= f(\theta(t))$ we have therefore that
\[
|\gamma(t)-\gamma(s)|= |t-s|
\]
whenever $|t-s|< \varepsilon$. Thus $\gamma(I)$ is a line segment for every interval $I\subset [0,2\pi]$ of length $\ell(I)<2\pi$,
which implies that $\gamma([0,2\pi))$ is a nondegenerate line segment. But on the other hand one must have
$\gamma(0)=\gamma(2\pi)$, which is a contradiction.
\end{example}

\section{Variational setting}\label{sec_varDistreconstr}

From now on we assume $M\subset\R^n$ to be a 
compact set connected by rectifiable arcs and equipped with the geodesic distance
\[
d_M(u,v):=\inf\left\{ \ell(\theta)\colon \theta\colon [0,1]\to M, \theta(0)=u, \theta(1)=v\right\},
\] 
where $\ell(\theta)$ denotes the Euclidean length of $\theta$.
Let $\Sigma_k\subset M$ be a sequence of closed sets. 

Given an $\varepsilon>0$ and a $k\in \N$, we define the functionals
\[
F_{\varepsilon, k}\colon C(M;\R^n)\to \R, F_\varepsilon\colon C(M;\R^n)\to \R, 
\]
by the formulae
\begin{equation}\label{eq_defFeps}
\begin{aligned}
F_{\varepsilon, k}(f) &:=\sup\left\lbrace  \left| \dfrac{|f(x)-f(y)|^2}{d_M^2(x,y)} -1 \right| \colon \{x,y\}\subset \Sigma_k, 
0<d_M(x,y)\leq \varepsilon 
\right\rbrace. \\
F_{\varepsilon}(f) &:=\sup\left\lbrace  \left| \dfrac{|f(x)-f(y)|^2}{d_M^2(x,y)} -1 \right| \colon \{x,y\}\subset M, 
0<d_M(x,y)\leq \varepsilon 
\right\rbrace. 
\end{aligned}
\end{equation}

The scope of this section is to prove the following easy result.

\begin{theorem}\label{th_Vdistest1}
	Let $M\subset \R^n$ be compact and connected by rectifiable arcs,
	and there exist $\varepsilon_0>0$, $C_1>0$
	such that
	\begin{equation}
	\label{eq_curv2}
     d_M(x,y) \leq |x-y| + C_1 |x-y|^2 
	\end{equation}	
	for all $(x,y)\in M\times M$ satisfying $d_M(x,y)\leq \varepsilon_0$.
	Denote
	\begin{equation}
\label{eq_barC2}
	\bar C_2:=\inf\left\{
	 \frac{|x-y|}{d_M(x,y)}\colon (x,y)\in M\times M, x\neq y
	 \right\}.
	\end{equation}
For an $x_0\in M$, $\Sigma_k\subset M$ a sequence of closed sets 
	satisfying
	$\Sigma_k\to M$ as $k\to\infty$ in the sense of the Hausdorff distance, and a $C_2\in (0, \bar C_2]$ set 
\begin{align*}
	\mathcal{C}_k &:= \left\lbrace f\in C(M;\R^n), f(x_0)=0, |f(x)-f(y)|\geq C_2 d_M(x,y)\, \mbox{for all}\, \{x,y\}\subset \Sigma_k \right\rbrace,\\
	\mathcal{C} &:= \left\lbrace f\in C(M;\R^n), f(x_0)=0, |f(x)-f(y)|\geq C_2 d_M(x,y)\, \mbox{for all}\, \{x,y\}\subset M \right\rbrace.
\end{align*}	Then the following asertions hold true.
\begin{itemize}
	\item[(i)] 
The variational problems
	\[
	\min \left\lbrace F_{\varepsilon,k}(f)\colon f\in \mathcal{C}_k \right\rbrace 
	\eqno{(P_k)}
	\]	
	have solutions for all $k\geq \bar k$, where $\bar k \in \N$ depends only on $\varepsilon$,  
	\item[(ii)]  If $f_k$ is a solution to $(P_k)$, then there is a subsequence of $\{f_k\}$ (not relabeled)
	such that $\lim_k f_k=f$ in the sense of uniform convergence, where $f$ solves
	\[
	\min \left\lbrace F_{\varepsilon}(f)\colon f\in \mathcal{C} \right\rbrace. 
	\eqno{(P)}
	\]	
	Note that if $f=\lim_k f_k$, then $f_k(M)\to \Sigma$, where $\Sigma:= f(M)$, in the sense of Hausdorff distance as $k\to +\infty$.
\item[(iii)]
    Moreover, every $f$ solving $(P)$ with $\varepsilon<\varepsilon_0$ 
    satisfies
	\begin{equation}\label{eq_Vlocdistest1sig1}
	d_M(x,y)(1-2C_1\varepsilon)\leq |f(x)-f(y)|\leq d_M(x,y)(1+2C_1\varepsilon), 
	\end{equation}
	if $d_M(x,y)\leq \varepsilon$, and 
	\begin{equation}\label{eq_Vdistest1sig2c}
	d_M (x,y)(1-2C_1\varepsilon)\leq d_\Sigma(f(x),f(y))\leq d_M (x,y)(1+2C_1\varepsilon)
	\end{equation}
	for all  $(x,y)\in M\times M$,
	where $d_\Sigma$ stands for the geodesic distance in $\Sigma$, i.e.\
	\[
	d_\Sigma(u,v):=\inf\left\{ \ell(\theta)\colon \theta\colon [0,1]\to \Sigma, \theta(0)=u, \theta(1)=v\right\}.
	\] 
\end{itemize}
\end{theorem}

Before proving the above theorem, we make a series of remarks.

\begin{remark}\label{rm_C2est1}
Under conditions of the above Theorem~\ref{th_Vdistest1} 
one has necessarily 
\[
\bar C_2\geq\frac{1}{1+ C_1\varepsilon_0}\wedge \frac{\varepsilon_0}{D} >0,
\]
where $D$ stands for the intrinsic diameter of $M$. 
In fact, since 
$|x-y|\leq d_M(x,y)$, the estimate~\eqref{eq_curv2} ensures 
	\begin{equation*}
\label{eq_curv3}
d_M(x,y) \leq (1+ C_1\varepsilon_0)|x-y|, 
\end{equation*}	
hence 
\[
\frac{|x-y|}{d_M(x,y)} \geq \frac{1}{1+ C_1\varepsilon_0}
\]
for all $(x,y)\in M\times M$ such that $d_M(x,y)<\varepsilon_0$, and 
\[
	 \frac{|x-y|}{d_M(x,y)} \geq \frac{\varepsilon_0}{D}
\]
for all $(x,y)\in M\times M$ such that $d_M(x,y)\geq \varepsilon_0$.
\end{remark}

\begin{remark}\label{rm_reach1}
The conditions on $M$ of the above Theorem~\ref{th_Vdistest1} are automatically satisfied if $M\subset \R^n$ is a $C^{1,1}$ smooth compact Riemannian
submanifold. The condition~\eqref{eq_curv2} can be seen then as a bound on curvatures of $M$. In particular, in this case
$\alpha :=\mbox{Reach}\, M>0$, and therefore~\eqref{eq_curv2} is satisfied according to  Lemma~\ref{lm_locdistest1} for
$\varepsilon_0:=\alpha$, with $C_1$ as in this Lemma. Moreover, in this case the constant $\bar C_2$ may be estimated in terms of $\alpha$ and the intrinsic diameter of $M$ according to Lemma~\ref{lm_locdistest2}. 
\end{remark}

\begin{remark}\label{rm_lowest1}
	Clearly, problem $(P)$ as well as approximating problems $(P_k)$ have many solutions. This is in the very nature of the problem
	statement: the given data are just intrinsic distances (which themselves do not contain any information on the embedding) and only very weak
	structural information on the embedding given by the constants $C_1$ and $C_2$ (and also by $\varepsilon_0$), so that if $M$ is, say,  a unit line segment, among solutions to $(P)$ there are infinitely many other embeddings of $M$ in a given Euclidean space as curves of unit length. However, they must be ``twisted not too much'', since 
any map $f$ solving $(P)$ is required to satisfy
	\begin{equation}\label{eq_locdistest1sig1_isom2}
	|f(x)- f(y)| \geq C_2d_M(x,y)
	\end{equation}
	for all $(x, y)\subset M\times M$. This, in particular, yields that the Euclidean diameter of $\Sigma:=f(M)$  cannot be arbitrarily small and thus excludes ``pathological'' embeddings like those provided by the Nash-Kuiper theorem.
	The fact that one requests the information on the structural constant $C_2$ to be retained by the embedding $f$ solving~$(P)$ via the requirement~\eqref{eq_locdistest1sig1_isom2}, besides avoiding such pathologies, is also used to force the injectivity of $f$, which is in a certain sense unavoidable (see Remark~\ref{rem-injectf1}). On the other hand, we do not force the embedding $f$ to satisfy the curvature-type estimate~\eqref{eq_curv2}. The reason is that in this way we are able to obtain a particularly simple algorithm to solve the approximating problems~$(P_k)$ (and hence to approximate embeddings  solving~$(P)$) based on solving a semidefinite programming problem. One might of course request more from the embedding a priori, but this would result in introducing  more constraints in the optimization problems and hence to substantially more complicated algorithms.
\end{remark}

\begin{proof}[Proof of Theorem~\ref{th_Vdistest1}]
	The proof will be divided into several steps.

{\em Step 1}. We first show that the sublevels of functionals  $F_{\varepsilon,k}$ are equicompact, that is, 
there is some  $\bar k\in \N$ depending only on $\varepsilon$ such that
the set 
\[
D_{K\varepsilon}:=\bigcup_{k\geq \bar k}\{f\in\mathcal{C}\colon F_{\varepsilon,k}(f)\leq K\} \subset C(M;\R^n)
\]
is compact for every $K>0$. In fact, by Lemma~\ref{lm_VequiLip1} for every $\varepsilon>0$ there is a $\bar k\in \N$ depending only on $\varepsilon$ such that for every $k\geq \bar k$ all $f\in C(M;\R^n)$ satisfying $F_{\varepsilon,k}(f)\leq K$ have equibounded Lipschitz constants over $\Sigma_k$. Hence, up to redefining each $f$ over $M\setminus \Sigma_k$ as an extension  from $\Sigma_k$ to $M$ with minimum Lipschitz constant, one has that
	the Lipschitz constants $\mathrm{Lip}\, f$ over $M$ are equibounded for all such $f$. Therefore, 
	the set $D_{K,\varepsilon}$
	is compact by Ascoli-Arzel\`{a} theorem\footnote{We retain the Italian tradition of ordering the manes of the authors of this famous theorem. In fact, it seems that it was discovered first by G.~Ascoli and later generalized in a separate work by C.~Arzel\`{a}.}
	as claimed.

{\em Step 2}. Note that the classes $\mathcal{C}\subset \mathcal{C}_k\subset C(M;\R^n)$ are closed. 
Since each functional $ F_{\varepsilon,k}$ is lower semicontinuous (as a supremum of a family of continuous functionals), the claim~(i) of the theorem being proven
(i.e.\ existence of solutions to problems~$(P_k)$ follows from compactnesss of sublevels of each $ F_{\varepsilon,k}$ with $k\in \N$ sufficiently large (depending only on $\varepsilon$).

{\em Step 3}.	Denote now
\begin{align*}
\hat{F}_{\varepsilon,k} (f) :=\left\{ 
\begin{array}{rl}
F_{\varepsilon,k} (f), & f\in \mathcal{C}_k,\\
+\infty, &\mbox{otherwise},	
\end{array}
\right. \qquad  	  
\hat{F}_{\varepsilon} (f) :=\left\{ 
\begin{array}{rl}
F_{\varepsilon} (f), & f\in \mathcal{C},\\
+\infty, &\mbox{otherwise}.	
\end{array}
\right.  
\end{align*}
Since the classes $\mathcal{C}\subset C(M;\R^n)$ and $\mathcal{C}_k\supset \mathcal{C}$ are closed, then the functionals $\hat{F}_{\varepsilon,k}$ also have equicompact sublevels.	

Observe that the equality $f=\lim_k f_k$, where $f_k\in \mathcal{C}_k$, yields that $f\in \mathcal{C}$.
In fact, for every $\{x,y\}\subset M$ there are $\{x_k,y_k\}\subset \Sigma_k$ such that 
$\lim_k x_k =x$, $\lim_k y_k =y$. Thus,
\[
|f(x)-f(y)|=\lim_k |f_k(x_k)-f_k(y_k)|\geq C_2 \lim_k  d_M(x_k, y_k)  = C_2 d_M(x, y),
\] 
showing that $f\in \mathcal{C}$.
Thus by Lemma~\ref{lm_VGammaconv1} one has
that $\gammalim_k \hat{F}_{\varepsilon, k}=\hat{F}_{\varepsilon}$.
Let now for each sufficiently large $k\in \N$ the function $f_k\in \mathcal{C}_k$ stand for a minimizer of $F_{\varepsilon,k}$ over $\mathcal{C}_k$, i.e.\ a solution to $(P_k)$. 
	Clearly, $f_k$ is also a minimizer of $\hat{F}_{\varepsilon,k}$ over the whole space $C(M;\R^n)$. Hence
by the main property of 
	$\Gamma$-convergence (theorem~2.10 from~\cite{bra06-Gammaconv})
		 one has that $f_k$ converge,
	up to a subsequence (not relabeled) 
	to a minimizer $f$ of $\hat{F}_\varepsilon$, hence a solution to 
	$(P)$. Note that when $\lim_k x_k=x$ in $M$, then $\lim_k f_k(x_k)=f(x)$ in view of the
	uniform convergence of $f_k$, which implies that $f_k(M)\to \Sigma$ in the sense of Hausdorff distance as $k\to +\infty$.
This proves claim~(ii) of the theorem.

{\em Step 4}.	Finally, let $f$ be any solution to $(P)$ with $\varepsilon<\varepsilon_0$. 
	The class $\mathcal{C}\subset C(M;\R^n)$
	contains the translation map
	$x\mapsto x-x_0$, and therefore we may invoke
	Lemma~\ref{lm_Vlocdistest3} to get~\eqref{eq_Vlocdistest1sig1}
	for $d_M(x,y)<\varepsilon$, and then Lemma~\ref{lm_Vlocdistest3a} (with $C:=1+2C_1\varepsilon$, $c:=1-2C_1\varepsilon$)
	to get~\eqref{eq_Vdistest1sig2c}.  
\end{proof}

The following technical assertions have been used in the above proof.

\begin{lemma}\label{lm_VGammaconv1}
	Let $\mathcal{C}\subset \mathcal{C}_k\subset C(M;\R^n)$ 
	be closed sets of maps, such that if $f_k\in \mathcal{C}_k$ and $f=\lim_k f_k$, then $f\in \mathcal{C}$.
	One has then
\[
	\gammalim_k (F_{\varepsilon, k} + \chi_{\mathcal{C}_k}) =F_{\varepsilon}+\chi_{\mathcal{C}},
\]
where
$F_{\varepsilon,k}, F_{\varepsilon} \colon C(M;\R^n)\to [0, +\infty]$
are defined by~\eqref{eq_defFeps}. 
\end{lemma}

\begin{proof}
One has
\begin{equation}\label{eq_Gconv1}
\begin{aligned}
	 &
	 \liminf_k F_{\varepsilon, k} (f_k)\\
	 &= 
	 \liminf_k  \sup\left\lbrace  \left| \dfrac{|f_k(x)-f_k(y)|^2}{d_M^2(x,y)} -1 \right| \colon \{x,y\}\subset \Sigma_k,
	 0 <d_M(x,y)\leq \varepsilon 
	 \right\rbrace \\
	 &\geq
	\liminf_k\sup\left\lbrace  \left| \dfrac{|f_k(x)-f_k(y)|^2}{d_M^2(x,y)} -1 \right| \colon \{x,y\}\subset \Sigma_k,
	\rho \leq d_M(x,y)\leq \varepsilon 
	\right\rbrace 
\end{aligned}
\end{equation}	
	for every $\rho>0$, the latter inequality being due to the fact that
	\[
	\Big\lbrace  \{x,y\}\subset \Sigma_k,
	\rho <d_M(x,y)\leq \varepsilon  \Big\rbrace \subset 	\Big\lbrace  \{x,y\}\subset \Sigma_k,
	0 <d_M(x,y)\leq \varepsilon  \Big\rbrace. 
	\]
	Consider arbitrary $\bar\varepsilon \in (\rho, \varepsilon)$, $\bar\rho\in (\rho,\bar \varepsilon)$, and 
	let $\{\tilde x, \tilde y\}\in M$ be such that $\rho<\bar \rho \leq d_M(\tilde x, \tilde y)\leq \bar \varepsilon<\varepsilon$  
and 	 
\begin{equation}\label{eq_xbarybar1}
	\left| \dfrac{|f(\tilde x)-f(\tilde y)|^2}{d_M^2(\tilde x,\tilde y)} -1 \right| =\sup\left\lbrace  \left| \dfrac{|f(x)-f(y)|^2}{d_M^2(x,y)} -1 \right| \colon \{x,y\}\subset M,
	\bar \rho \leq d_M(x,y)\leq \bar \varepsilon 
	\right\rbrace.
\end{equation}
Let $\{x_k, y_k\}\in \Sigma_k$
be such that $\lim_k x_k= \tilde x$, $\lim_k y_k= \tilde y$. Then
\begin{equation}\label{eq_Gconv2}
	\begin{aligned}
		\liminf_k &\sup\left\lbrace  \left| \dfrac{|f_k(x)-f_k(y)|^2}{d_M^2(x,y)} -1 \right| \colon \{x,y\}\subset \Sigma_k,
		\rho \leq d_M(x,y)\leq \varepsilon 
		\right\rbrace \\
		&\geq \liminf_k  \left| \dfrac{|f_k(x_k)-f_k(y_k)|^2}{d_M^2(x_k,y_k)} -1 \right|  
		\quad\mbox{since $d_M(x_k,y_k)\in [\rho,\varepsilon]$ for large $k$}\\ 
		&= \left| \dfrac{|f(\tilde x)-f(\tilde y)|^2}{d_M^2(\tilde x,\tilde y)} -1 \right|  
		\\ 
		& =\sup\left\lbrace  \left| \dfrac{|f(x)-f(y)|^2}{d_M^2(x,y)} -1 \right| \colon \{x,y\}\subset M,
		\bar \rho \leq d_M(x,y)\leq \bar \varepsilon 
		\right\rbrace \quad\mbox{by~\eqref{eq_xbarybar1}},
	\end{aligned}
\end{equation}	
Thus, combining~\eqref{eq_Gconv1} and~\eqref{eq_Gconv2}, we get
\[
	 	\liminf_k  F_{\varepsilon, k} (f_k)
		\geq \sup\left\lbrace  \left| \dfrac{|f(x)-f(y)|^2}{d_M^2(x,y)} -1 \right| \colon \{x,y\}\subset M,
	 	\bar \rho \leq d_M(x,y)\leq \bar \varepsilon 
	 	\right\rbrace
\]
for every $\rho>0$. Taking  in the above estimate the supremum with respect to
$\bar\rho,\bar\varepsilon$, such that $0<\bar\rho<\bar\varepsilon< \varepsilon$, we obtain that
\begin{align*}
	\liminf_k  F_{\varepsilon, k} (f_k)
	&
	\geq \sup\left\lbrace  \left| \dfrac{|f(x)-f(y)|^2}{d_M^2(x,y)} -1 \right| \colon \{x,y\}\subset M,
	0 < d_M(x,y) < \varepsilon 
	\right\rbrace
	\\&
	= F_{\varepsilon}(f),
\end{align*} 
which means
\[
	\gammaliminf_k F_{\varepsilon, k} \geq F_{\varepsilon}.
\]
The last inequality, together with the fact that $\mathcal{C}_k \supset \mathcal{C}$, yields that
\[
	\gammaliminf_k (F_{\varepsilon, k}+\chi_{\mathcal{C}_k})\geq F_{\varepsilon} +\chi_{\mathcal{C}}.
\]

The inequality
\[
	\gammalimsup_k (F_{\varepsilon, k}+\chi_{\mathcal{C}_k}) \leq F_{\varepsilon}+\chi_{\mathcal{C}}
\]
is immediate since $F_{\varepsilon, k}(f) \leq F_{\varepsilon}(f)$ for every $f\in C(M;\R^n)$, and 
$\mathcal{C}\subset \mathcal{C}_k$. 
This concludes the proof.
\end{proof}

\begin{lemma}\label{lm_VequiLip1}
If $M\subset\R^n$ is compact and connected by rectifiable arcs, then for every $\varepsilon>0$ there is a $\bar k\in \N$ (depending only on $\varepsilon$) such that for every $k\geq \bar k$ and
for every $f\in C(M;\R^n)$ satisfying $F_{\varepsilon,k}(f)\leq K$ one has
\begin{equation}\label{eq_Lipest_Fk1}
|f(x)-f(y)| \leq C d_M(x,y)
\end{equation}
for all $\{x,y\}\in \Sigma_k$, $x\neq y$, where $C>0$ depends only on $K$.  
\end{lemma}

\begin{proof}
	Since the functionals $F_{\varepsilon,k}$ are nonnegative, we may assume $K>0$. 
	If $f\in C(M;\R^n)$ satisfies $F_{\varepsilon,k}(f)\leq K$, then 
\begin{equation}\label{eq_VLipf_gener0}
\begin{aligned}
	\frac{|f(x)-f(y)|}{d_M(x,y)} \leq C
\end{aligned}	
\end{equation}
	for all $\{x,y\}\in \Sigma_k$, $0< d_M(x,y)<\varepsilon$ (with $C:=\sqrt{K+1}$). This proves~\eqref{eq_Lipest_Fk1} for such couples
	$\{x,y\}$. 

	To prove~\eqref{eq_Lipest_Fk1} for the remaining couples $\{x,y\}\in \Sigma_k$, consider a finite $\varepsilon$-net $\{x_j\}_{j=1}^N$ of $M$. For each pair of indices $(i, j) \in \{1,\ldots, N\}^2$, let us do the following.
    \begin{itemize}
    	\item Let us fix a geodesic $\theta_{ij}$ in $M$, connecting $x_i$ to $x_j$ (i.e.\ $\ell(\theta_{ij})=d_M(x_i, x_j)$) and parameterized for convenience over $[0,1]$.
    	\item Let $0 = t_0 < t_1 < \ldots < t_{m(i, j)} = 1$ be such a partition of $[0, 1]$ that
	\[
		d_M \big( \theta_{ij}(t_{m - 1}), \theta_{ij}(t_m) \big) \leq \frac{\varepsilon}2
		\quad \text{for all $m \in \{1, \dots, m(i, j)\}$}.
	\]
     	\item Finally, let
	\[
		\mathcal{V}_{ij} =  \{v^{m} = \theta_{ij}(t_m) : 1 \leq m \leq m(i, j) \}
	\]
	stand for the set of corresponding points on $\theta_{ij}$ (including $0$ and $1$, so that $\mathcal{V}_{ij}$ contains both $x_i$ and $x_j$). Note that there is a natural order of points in $\mathcal{V}_{ij}$. Namely, for $\{\theta_{ij} (s), \theta_{ij} (t)\} \subset \mathcal{V}_{ij}$ we may write
     	$\theta_{ij} (s) \leq  \theta_{ij} (t)$, if $s\leq t$. We will say further that $u\in\mathcal{V}_{ij}$ and  $v\in\mathcal{V}_{ij}$ are two consecutive points, if there are no points between $u$ and $v$ in the sense of the introduced order. 
   \end{itemize}    

Choose a $\bar k\in \N$ such that for every $k\geq \bar k$ and for pair of indices $i, j=1,\ldots, N$ and every $l=1,\ldots m(i,j)$ 
there exists a  $v_k^{l}\in \Sigma_k$ with
\begin{align*}
d_M(v_k^{l},v^{l}) &<\frac{\varepsilon}{2},\\ 	
d_M(v_k^{l},v_k^{l+1}) & < 2d_M(v^{l},v^{l+1}),\quad  \mbox{when } l<m(i,j).
\end{align*}
The latter inequality implies, in particular, that $d_M(v_k^{l},v_k^{l+1}) < \varepsilon$.
Connecting each $v_k^{l}$ with $v_k^{l+1}$ by a geodesic segment, we get
a ``polygonal line'' $\sigma_{ij}$ made of geodesic segments with 
vertices  at most $\varepsilon/2$ close (in $d_M$) to the respective points of $\mathcal{V}_{ij}$,
geodesic distance between consecutive vertices at most $\varepsilon$, and, finally,
\[
\ell(\sigma_{ij})< 2\ell(\theta_{ij})=2d_M(x_i, x_j).
\] 
We have then 
\begin{equation}\label{eq_VLipf_gener1}
\begin{aligned}
|f(\sigma_{ij}(0))-f(\sigma_{ij}(1))| &= |f(v_k^{0})-f(v_k^{m})| \leq \sum_{l=0}^{m-1} |f(v_k^{l})-f(v_k^{l+1})|\\
& \leq C \sum_{l=0}^{m-1} d_M(v^{l},v^{l+1})
= C\ell(\sigma_{ij})< 2 C d_M(x_i, x_j).
\end{aligned}	
\end{equation}
	Finally, for arbitrary $\{x, y\}\in \Sigma_k$, $d_M(x,y)\geq\varepsilon$, we find a couple $\{x_i, x_j\}$ such that
	\[
	d_M(x_i,x) <\varepsilon,\quad 	d_M(x_j,y) <\varepsilon, 
	\]
	and estimate
	\begin{equation}\label{eq_VLipf_gener2}
	\begin{aligned}
|f(x)-f(y)| &\leq 
	|f(x)-f(\sigma_{ij}(0))|+ |f(\sigma_{ij}(0))-f(\sigma_{ij}(1))| 
	+|f(\sigma_{ij}(1))-f(y)|\\
	&<  C d_M(x,\sigma_{ij}(0)) +2 Cd_M(x_i, x_j) + C d_M(\sigma_{ij}(1),y)
	\end{aligned}	
	\end{equation}
	in view of~\eqref{eq_VLipf_gener1}. 
	But
	\begin{equation}\label{eq_VLipf_gener3a}
	\begin{aligned}
	 d_M(x,\sigma_{ij}(0))  & \leq d_M(x, x_i) + d_M(x_i,\sigma_{ij}(0)) \leq 3\varepsilon/2,
	\end{aligned}	
	\end{equation}
	and analogously
	\begin{equation}\label{eq_VLipf_gener3b}
	\begin{aligned}
	d_M(y,\sigma_{ij}(1))  & \leq  3\varepsilon/2,
	\end{aligned}	
	\end{equation}
	while
	\begin{equation}\label{eq_VLipf_gener3c}
	\begin{aligned}
	d_M(x_i, x_j) & \leq d_M(x_i,x) + d_M(x,y) + d_M(y, x_j)\leq 2\varepsilon + d_M(x,y).
	\end{aligned}	
	\end{equation}
	Plugging~\eqref{eq_VLipf_gener3a},~\eqref{eq_VLipf_gener3b} and~\eqref{eq_VLipf_gener3c} into~\eqref{eq_VLipf_gener2}, we get
	\begin{equation}\label{eq_VLipf_gener4}
	\begin{aligned}
	|f(x)-f(y)| &<  C (11\varepsilon+ 2d_M(x,y)) \leq 13 C d_M(x,y),
	\end{aligned}	
	\end{equation}
	because $d_M(x,y)\geq\varepsilon$. Together with~\eqref{eq_VLipf_gener0} the estimate~\eqref{eq_VLipf_gener4} proves the claim. 
\end{proof}

\begin{lemma}\label{lm_Vlocdistest3}
	Let $M$ be as in Theorem~\ref{th_Vdistest1}.
	If $f\in C(M;\R^n)$ is a minimizer of $F_{\varepsilon}$ with 
	$\varepsilon\leq \varepsilon_0$
	over some class containing some rigid translation, then
	\[
	d_M(x,y)(1-2C_1\varepsilon)\leq |f(x)-f(y)|\leq d_M(x,y)(1+2C_1\varepsilon), 
	\]
	if $d_M(x,y)<\varepsilon$. The same holds 
	for $\{x,y\}\subset \Sigma_k$
	when $f\in C(M;\R^n)$ is a minimizer of $F_{\varepsilon,k}$.
\end{lemma}

\begin{remark}\label{rm_reach1b}
	In view of Remark~\ref{rm_reach1} the conditions on $M$ are automatically satisfied if $M\subset \R^n$ is a $C^{1,1}$ smooth compact
	submanifold, with $\varepsilon_0$ and $C_1$ in this case being as in Lemma~\ref{lm_locdistest1}.  
\end{remark}

\begin{proof}[Proof of Lemma \ref{lm_Vlocdistest3}.]
	Note that $F_{\varepsilon}(f)$ is invariant with respect to the compositions of $f$ with rigid translations, and in particular the value of
	$F_{\varepsilon}$ over any rigid translation is equal to that over the identity map $\mathrm{id}$.
	The relationship~\eqref{eq_curv2}
	implies
	\[
	\frac{1}{(1+C_1\varepsilon)^2} \leq \dfrac{|x-y|^2}{d_M^2(x,y)}\leq 1 
	\]
	whenever $d_M(x,y)<\varepsilon$.
	One has therefore for such couples $(x,y)\in M\times M$ the estimate
	\[
		F_{\varepsilon}(f) \leq F_{\varepsilon}(\mathrm{id}) \leq 1-\frac{1}{(1+C_1\varepsilon)^2} \leq 2 C_1 \varepsilon,
	\]
	showing the statement for $F_{\varepsilon}$. The proof for $F_{\varepsilon,k}$ is identical.
\end{proof}

\begin{lemma}\label{lm_Vlocdistest3a}
	Suppose that for some 
	$C>0$ and $f\in C(M;\R^n)$ one has
\begin{equation}\label{eq_Vdistest1sig20}
	|f(x)-f(y)| \leq C d_M(x,y), 
\end{equation}
	if $d_M(x,y)<\varepsilon$. 
	Then
	\begin{equation}\label{eq_Vdistest1sig2a}
	d_\Sigma(f(x),f(y))\leq C d_M(x,y),
	\end{equation}
	where $\Sigma:= f(M)$. If, moreover, $f$ is injective function with a continuous inverse $f^{-1}\colon \Sigma \to M$  (which is the case, e.g., when $f$ is proper), and
\begin{equation}\label{eq_Vdistest1sig200}
c d_M(x,y)\leq 
|f(x)-f(y)|
\end{equation}
for some $c>0$,
if $d_M(x,y)<\varepsilon$,
then also
	\begin{equation}\label{eq_Vdistest1sig2b}
	c d_M (x,y)\leq d_\Sigma(f(x),f(y)). 
	\end{equation}
\end{lemma}

\begin{remark}\label{rem-injectf1}
	The estimate~\eqref{eq_Vdistest1sig2b} cannot hold for $f$ just satisfying~\eqref{eq_Vdistest1sig20} 
	and~\eqref{eq_Vdistest1sig200} only for $d_M(x,y)<\varepsilon$,
	unless $f$ is injective, as can be seen from the example of
	$M$ a line segment of length $2\pi$ (identified with $[0,2\pi]$) and $f\colon [0,2\pi]\to \R^2$ defined by
	$f(t):=(\cos t, \sin t)$. 
\end{remark}

\begin{proof}
	If $\theta\colon [0,1]\to M$ is a Lipschitz curve, then~\eqref{eq_Vdistest1sig20} implies that
	 so is $\sigma:=f\circ\theta\colon [0,1]\to \R^n$ and its metric derivative $|\dot\sigma|$ satisfies
\begin{equation}\label{eq_Vmetrdir1}
	|\dot\sigma|(t)\leq C |\dot\theta|(t) 
\end{equation}
for a.e.\ $t\in [0,1]$.
Then~\eqref{eq_Vmetrdir1} gives~\eqref{eq_Vdistest1sig2a}. In fact, if $\theta$ is a geodesic curve
connecting $x=\theta(0)$ to $x=\theta(0)$, then 
\[
d_\Sigma(f(x),f(y))\leq \int_0^1 |\dot\sigma|(t)\, dt \leq 
C\int_0^1 |\dot\theta|(t)\, dt =
Cd_M(x,y).
\]

If $f$ is injective, then take an arbitrary $\delta>0$, and consider a rectifiable curve
$\sigma\colon [0,1]\to f(M)\subset \R^n$ such that
$\sigma(0)=f(x)$, $\sigma(1)=f(y)$ with
\[
\int_0^1 |\dot\sigma|(t)\, dt \leq d_\Sigma(f(x),f(y)) +\delta.
\]
Denote $\theta(t):= f^{-1}(\sigma(t))$ for every $t\in [0,1]$. If $f^{-1}$ is continuous, then so is $\theta$, and hence
for every $t\in [0,1]$ one has $d_M(\theta(t),\theta(t+s))<\varepsilon$ once $s$ is sufficiently small.
Thus from~\eqref{eq_Vdistest1sig2a} we get
\begin{equation}\label{eq_Vmetrdir2}
c |\dot\theta|(t) 
\leq 
|\dot\sigma|(t)
\end{equation}
for a.e.\ $t\in [0,1]$. Therefore,
\[
c d_M(x,y)\leq 
c\int_0^1 |\dot\theta|(t)\, dt \leq 
d_\Sigma(f(x),f(y)) +\delta,
\]
and taking the limit in the above inequality as $\delta\to 0^+$, we arrive at~\eqref{eq_Vdistest1sig2b}.
\end{proof}

\section{Discrete variational setting and algorithm}\label{ssec_SDP1}

	Let $\{y_i\}\subset M$ be a dense set in $M$, and denote for the sake of brevity
\[
d_{ij}:=d_M(y_i, y_j).
\]

Given an $\varepsilon>0$ and a $k\in \N$, we define the functional $F_{\varepsilon,k}\colon (\R^n)^k\to \R$
by the formula
\[
F_{\varepsilon,k}(x_1,\ldots, x_k):=\max\left\lbrace  \left| \dfrac{|x_i-x_j|^2}{d_{ij}^2} -1 \right| \colon i, j =1,\ldots, k, i\neq j, d_{ij}<\varepsilon 
\right\rbrace. 
\]
The following statement is just a direct application of Theorem~\ref{th_Vdistest1} to the sequence $\Sigma_k:=\{y_j\}_{j=1}^k\subset M$, once
we denote $x_j:= f(y_j)$ for all $j\in \N$, where $f$ is an embedding provided by Theorem~\ref{th_Vdistest1}(ii). 

\begin{proposition}\label{prop_distest4}
Let $M$, $C_1$, $C_2$ and $\varepsilon_0$ be as in Theorem~\ref{th_Vdistest1}.
Assume that $(x_i^k)_{i=1}^k\in (\R^n)^k$ be a minimizer of $F_{\varepsilon,k}$ with 
$\varepsilon< \varepsilon_0$
over
	the set $X^k  \subset (\R^n)^k$ defined by
\[
X^k:=\left\lbrace ((x_i)_{i=1}^k\in (\R^n)^k\colon |x_i-x_j|\geq C_2 d_{ij}, i,j=1,\ldots, k\right\rbrace.   
\]
Then up to a subsequence one has $x_i^k\to x_i$ as $k\to \infty$, and
\begin{equation}\label{eq_locdistest1sig1}
d_{ij}(1-2C_1\varepsilon)\leq |x_i-x_j|\leq d_{ij}(1+2C_1\varepsilon), 
\end{equation}
whenever $d_{ij}<\varepsilon$.  	
Further, 
\begin{equation}\label{eq_locdistest1sig2b}
	d_{ij}(1-2C_1\varepsilon)\leq d_\Sigma(x_i,x_j)\leq d_{ij}(1+2C_1\varepsilon)
\end{equation}	
for all $\{i,j\}\subset \N$.
In particular the statement is valid when $M$ is a $C^{1,1}$ smooth compact Riemannian submanifold of $\R^n$ with
$\varepsilon_0:=\mbox{Reach}\, M$, $C_1$ and $C_2$ are as in Lemmata~\ref{lm_locdistest1},~\ref{lm_locdistest2}.
\end{proposition}

\subsection*{Reduction to semidefinite programming problem}

The problem of minimizing $F_{\varepsilon,k}$  over
the set $X^k  \subset (\R^n)^k$ is written as minimizing the convex function $G_{\varepsilon,k}$ of a matrix 
defined by
\begin{align*}
G_{\varepsilon,k}(x_1,\ldots, x_k):=\max\left\lbrace  \left| \dfrac{K_{ii}+K_{jj}-2K_{ij} }{d_{ij}^2} -1 \right| \colon i, j =1,\ldots, k, i\neq j, d_{ij}<\varepsilon 
\right\rbrace. 
\end{align*}
over the set of positive semidefinite matrices $K$ satisying the set of convex constraints
\[
K_{ii}+K_{jj}-2K_{ij}\geq C_2^2 d_{ij}^2, i,j=1,\ldots, k.   
\]
The solution $K$ of the latter problem is the Gram matrix of a set of vectors $\{x_i\}$, i.e.\ $K_{ij}=\ x_i\cdot x_j$, $i,j=1,\ldots, k$, which minimize $F_{\varepsilon,k}$  over
$X^k$.

Adding a new scalar variable $t\in \R$ one reduces the above problem to the following \emph{semidefinite programming} problem (i.e.\ a problem of minimization of a linear function with linear constraints over the cone of positive semidefinite matrices), namely
\begin{equation}\label{eq_SDP1}
\begin{aligned}
\mbox{minimize} &\quad t\quad \mbox {over the pairs} \quad (t,K)\quad \mbox{subject to}\\
& -t d_{ij}^2\leq K_{ii}+K_{jj}-2K_{ij}-d_{ij}^2\leq t d_{ij}^2,\\
&\qquad\mbox{for all } 
 i, j =1,\ldots, k, i\neq j, d_{ij}<\varepsilon,\\
& K_{ii}+K_{jj}-2K_{ij}\geq C_2^2 d_{ij}^2,\\
&\qquad \mbox{for all } 
i, j =1,\ldots, k, i\neq j,\\
& K \quad \mbox{positive semidefinite $k\times k$ matrix}.
\end{aligned}
\end{equation}

\section{How to compute \v{C}ech cohomologies}\label{sec_Cech}

Since 
$M$ and $\Sigma:=f(M)$ are homeomorphic (even bilipschitz equivalent) by Theorem~\ref{th_Vdistest1}, they
	have the same homologies and cohomologies for every reasonable (co)homology theory. In topological data analysis it is quite usual to
	consider \v{C}ech cohomologies of $M$. Computing them when $M$ is not observed directly, but is just determined by distance matrices,
	one has to construct its embedding into $\R^n$ and build \v{C}ech complexes built on Euclidean balls centered at samples from 
	the image of such an embedding.
	We show here that for the embeddings $f\colon M\to \R^n$ 
	provided by Theorem~\ref{th_Vdistest1}(iii) one can give explicit estimates on such complexes (how small should be the radii of the balls and
	how well fitted should be the set of their centers) so as to get the cohomologies of $M$. Throughout this section we always denote by
	$B_r(x)\subset \R^n$ the open Euclidean ball of radius $r>0$ with center $x\in \R^n$.

\begin{lemma}\label{lm_Cech1cover1b}
	Assume that the conditions of  
	Theorem~\ref{th_Vdistest1} be satisfied.
	If $x\in \Sigma\cap B_\sigma(\bar x)$, $\bar x\in \Sigma$, and $x=f(y)$, $y\in M$, $\sigma\leq C_2\varepsilon_0$, then
	\[
		y\in B_{\sigma'}(\bar y),
		\quad
		\text{where $\bar x=f(\bar y)$ and $\sigma':=\left(1+C_1\frac{\sigma}{C_2}\right)\frac{\sigma}{C_2}$.}	
	\] 
	
	Vice versa, if $y\in M\cap B_{\sigma''}(\bar y)$, $\bar y\in M$, where 
	$\sigma''>0$ is defined by the equation
	\begin{equation}\label{eq_defsigma''}
	\left(1+C_1\frac{\sigma''}{C_2}\right)\frac{\sigma''}{C_2}=\sigma,
	\end{equation}
	and $\sigma''\leq C_2\varepsilon$, then
	for $x=f(y)$,  $\bar x=f(\bar y)$ one has $x\in \Sigma\cap B_\sigma(\bar x)$.
\end{lemma}

\begin{remark}
	It is worth noting that $\sigma''<\sigma<\sigma'$.
\end{remark}

\begin{proof}
	If $|x-\bar x|< \sigma$, then 
	\[
	d_M(y, \bar y)\leq \frac{\sigma}{C_2}
	\]
	in view of~\eqref{eq_locdistest1sig1_isom2},
	so that the the first claim follows from~\eqref{eq_curv2}.
	
	To prove the second claim, note that the inequality $0< C_2 \leq \bar C_2$ and the definition~\eqref{eq_barC2} of $\bar C_2$ implies the bound
	\[
		d_M(y, \bar y)\leq \frac{\sigma''}{C_2} \quad \text{for all $|y-\bar y| < \sigma''$.}
	\]
	Hence, it suffices to apply~\eqref{eq_Vlocdistest1sig1} 
	and recall the definition of $\sigma''$ (see~\eqref{eq_defsigma''}).
\end{proof}

Let $\Lambda\subset\N$ be a finite set of indices, such that 
for all 
$y\in M$
there is a $\lambda\in \Lambda$ and an
$\bar y_\lambda\in M$,
satisfying 
\[
d_M(y,\bar y_\lambda)\leq \delta
\]
for some $\delta>0$. In other words, $\{\bar y_{\lambda}\}$ is a finite $\delta$-net of
of $M$ (equipped with $d_M$).

Denote now by $C_\Sigma(r)$ the \v{C}ech complex  built on the Euclidean balls
$B_{r}(\bar x_{\lambda})$, where $\bar x_{\lambda}:=f(\bar y_{\lambda})$, 
and by $C_M(r)$ the \v{C}ech complex built on the Euclidean balls
$B_{r}(\bar y_{\lambda})$. We note that the vertices of all these complexes may be considered the same
(namely, the set of vertices of all them may be identified with the index set $\Lambda$).

\begin{lemma}\label{lm_Cech1cover1c}
Under the conditions of 
Theorem~\ref{th_Vdistest1} one has
\[
C_M(\sigma'')\subset C_\Sigma(\sigma) \subset C_M(\sigma')
\]
when  $0<\sigma\leq C_2\varepsilon_0$ and
	$\sigma''$ is defined by~\eqref{eq_defsigma''}
with $0<\sigma''\leq C_2\varepsilon_0$.
\end{lemma}

\begin{proof}
	Follows immediately from Lemma~\ref{lm_Cech1cover1b}.
\end{proof}

We now consider the particular case when $M\subset \R^n$ is a $C^{1,1}$ smooth compact Riemannian
submanifold.

\begin{proposition}\label{prop_Cech1a}
	Let $M\subset \R^n$ is a $C^{1,1}$ smooth compact Riemannian
	submanifold, so that
	$\alpha :=\mbox{Reach}\, M>0$, 
the constants $C_1$ be as in Lemma~\ref{lm_locdistest1} and $\bar C_2$ be as in Lemma~\ref{lm_locdistest2}.	
Assume  $C_2\in (0, \bar C_2]$ and
$0<\sigma\leq C_2\alpha$ to be so small that
\[
\sigma':=\left(1+C_1\frac{\sigma}{C_2}\right)\frac{\sigma}{C_2} < \frac 1 2 \sqrt{\frac 3 5}\alpha,
\]
$\sigma''$ defined by~\eqref{eq_defsigma''}
satisfy
\[
0<\sigma'' < C_2 \alpha 
\]
 and let
$\delta\in (0,\alpha)$
be such that
\[
\rho(\delta):=(1+C_1\delta)\delta \leq \sigma'.
\]
Then $H^*(C_\Sigma(\sigma);\R) \simeq H^*(M;\R)$, where $H^*$ stands for the
\v{C}ech cohomology. 
\end{proposition}

\begin{remark}
	One may take $\bar y_\lambda$ to be drawn by sampling $M$ in i.i.d. way according to the volume measure on $M$. 
	In fact by proposition~3.2 of~\cite{nsw08} one has then that if $\#\Lambda> n(M,\rho,p)$, then
	\[
	M\subset \bigcup_\lambda B_\rho (\bar y_\lambda) 
	\]  
	with probability  at least $1-p$
	and the number $n(M,\rho,p)$ depends explicitly, besides $\rho$ and $p$, also on the total volume and the dimension of $M$.
\end{remark}

\begin{proof}[Proof of Proposition \ref{prop_Cech1a}.]
Lemma~\ref{lm_locdistest1} implies~\eqref{eq_curv2}, so that one has 
	\[
	M\subset \bigcup_{\lambda} B_{\rho(\delta)} (\bar y_{\lambda}), 
	\] 
	since $\delta< \varepsilon_0$. Therefore,
	\[
	M\subset \bigcup_{\lambda} B_{\sigma'} (\bar y_{\lambda}), 
	\] 
	since $\sigma'> \rho(\delta)$.
    But 
    \[
    2\sigma''<2\sigma' < \sqrt{\frac 3 5}\lambda 
    \]	 
    implies $H^*(C_M(\sigma');\R) \simeq H^*((M)_{\sigma'};\R)$
    and $H^*(C_M(\sigma'');\R) \simeq H^*((M)_{\sigma''};\R)$ by proposition~3.1 of~\cite{nsw08}.
    Since $(M)_{\sigma''}\subset (M)_{\sigma'}\subset (M)_{\alpha}$, then both 
    $(M)_{\sigma''}$ and $(M)_{\sigma'}$ are clearly retractible to $M$, so that
\[    
H^*((M)_{\sigma''};\R) \simeq H^*(M;\R) \quad \mbox{and } H^*((M)_{\sigma'};\R) \simeq H^*(M;\R).
\]    Therefore, 
    \[
    H^*(C_M(\sigma');\R) \simeq H^*(M;\R)\simeq H^*(C_M(\sigma'');\R),
    \]
     and it suffices now to apply Lemma~\ref{lm_Cech1cover1c} with 	$\varepsilon_0:=\alpha$ to get the claim.
\end{proof}

\section{Numerical experiments}\label{sec_numericsDistrreconstr}

In this section we present numerical experiments illustrating the performance of our 
procedure for four sample datasets: a line segment, a two-dimensional sphere, the Swiss Roll, and the flat torus embedded in $\R^4$ as the Clifford torus.
In all the experiments, except for the last one, the constant $C_2$ from \eqref{eq_SDP1} is set to $2/\pi$. To reconstruct Clifford torus, we chose $C_2 = 0.2 / \pi$.
For quantitative measure of the performance, we compute the error
\begin{equation}
	\label{err}
	Err = \frac1k \sqrt{\sum\limits_{i, j = 1}^k \frac{\left(\widehat d_{ij} - d_{ij} 
	\right)^2}{d_{ij}^2}},
\end{equation}
which reflects the average relative error in pairwise distances.
Here $\widehat d_{ij}$, $1 \leq i,j \leq k$, stands for the pairwise distance between the 
recovered i-th and j-th elements of the sample.

We start with the example of a line segment.
We took $k = 100$ equidistant points on the unit interval $[0,1]$ and 
embedded them into $\R^2$ using our algorithm with $n=2$ and $\varepsilon=0.2$.
The result is shown in Figure~\ref{segment}.
Though the recovered points do not lie on a segment, the error~\eqref{err} is equal to 
$Err=9 \cdot 10^{-4}$, which is quite small.

\begin{figure}[h]
	\noindent
	\includegraphics[width=0.4\linewidth]{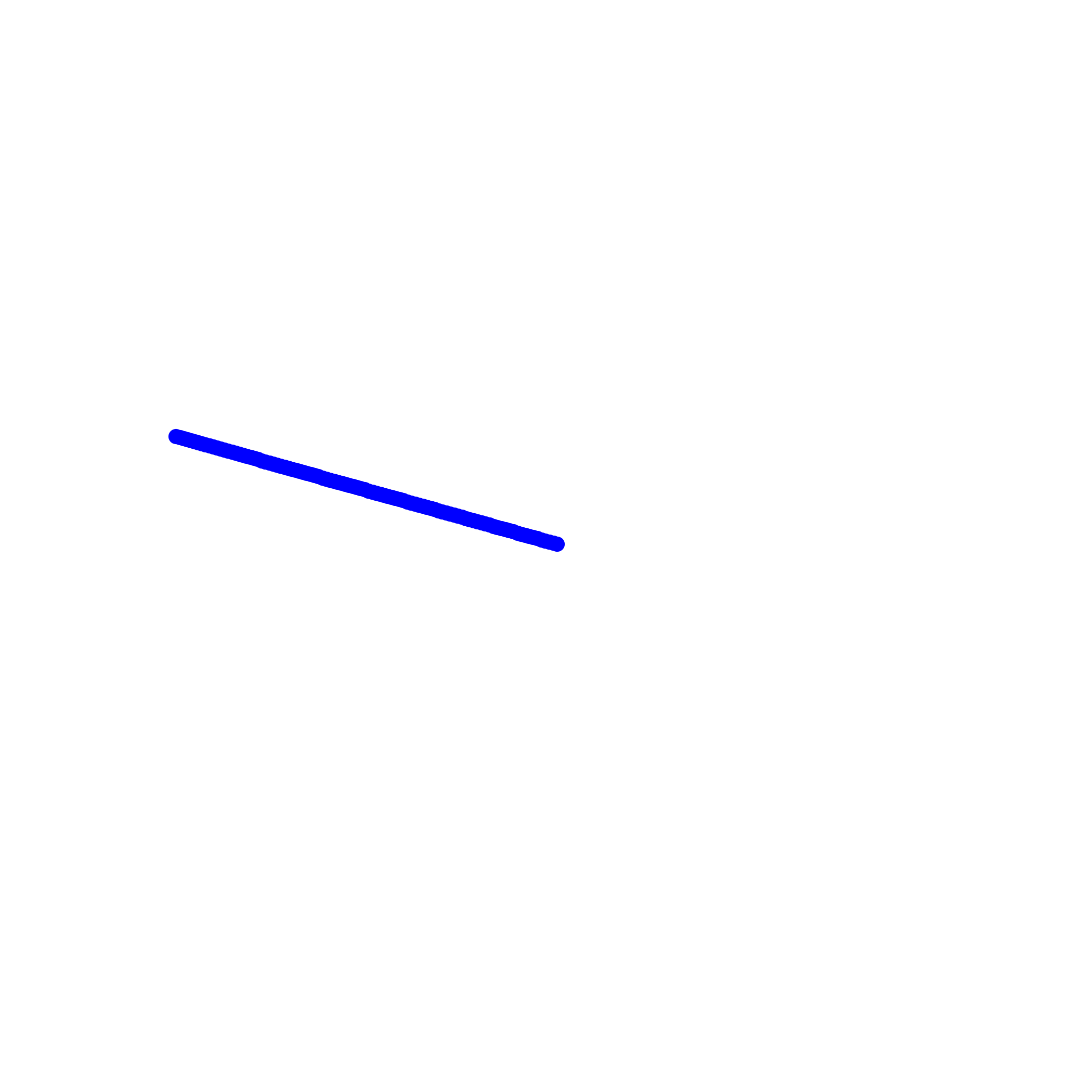}
	\caption{Reconstruction of a line segment from pairwise distances. The average 
	distance error is $9 \cdot 10^{-4}$.}
	\label{segment}
\end{figure}

In the example of a two-dimensional sphere we have two different setups.
In the first one, we took $k = 100$ points on a grid on unit sphere $S^2$, computed exact 
geodesic distances and applied the procedure with parameters $n = 3$ and 
$\varepsilon = 0.6$.
After that, we computed approximate pairwise geodesic distances over the resulting 
point cloud 
as it is done is Isomap.
As a result, we obtained $Err = 0.13$.
In the second setup, we had $k = 100$ points drawn independently from uniform 
distribution on the sphere and computed exact geodesic distances between them.
After that, we performed the embedding into $\R^3$ using our procedure with 
parameters $n = 3$ and $\varepsilon = 0.6$, and computed approximate pairwise 
geodesic distances between the embedded points again using the same method as in Isomap.
As a result, we obtained $Err = 0.16$.
The results of the sphere embedding are displayed in Figure~\ref{sphere}.

\begin{figure}[h]
	\noindent
	\includegraphics[width=\linewidth]{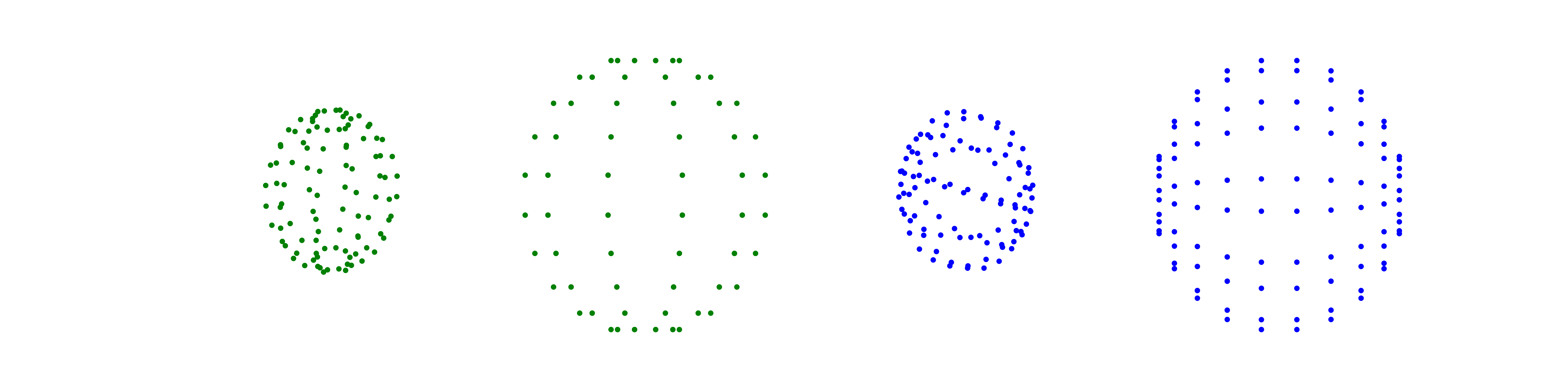}
	\includegraphics[width=\linewidth]{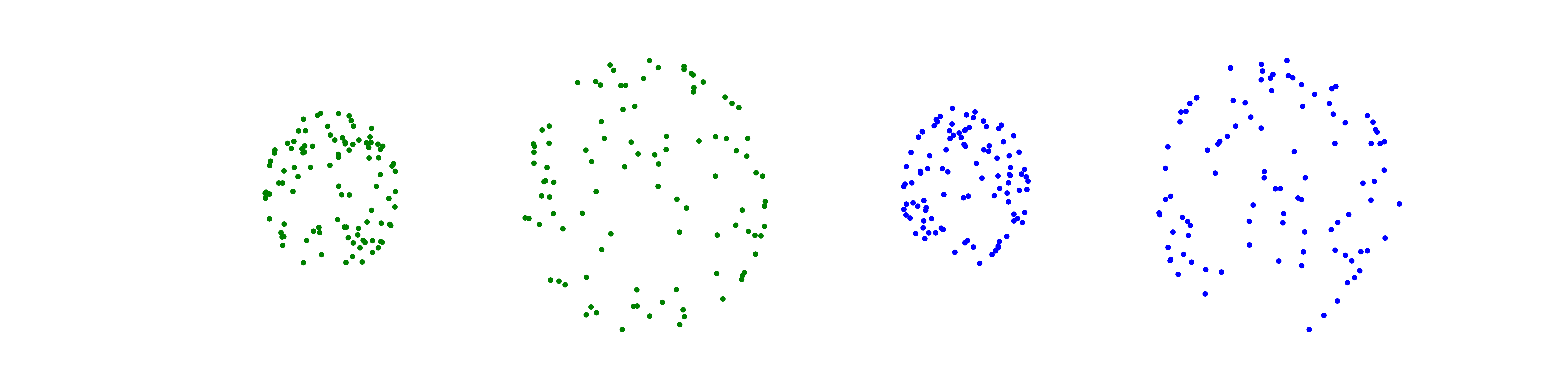}
	\caption{Reconstruction of a unit sphere from pairwise distances. Top line, columns 
		1 and 2: points on a grid on the sphere. Top line, columns 3 and 4: the recovered 
		points of the unit sphere. Bottom line, columns 1 and 2: points drawn from the 
		uniform distribution on unit sphere. Bottom line, columns 3 and 4: the recovered 
		points from approximate geodesic distances.}
	\label{sphere}
\end{figure}

Next, we carried out experiments on the widely known synthetic {\sc Swiss Roll} 
dataset from the Scikit-learn library in Python.
Here we also have two different setups.
In the first one, we generated $k=100$ points, computed pairwise Euclidean 
distances and applied the procedure with parameters $n = 3$ and $\varepsilon = 3$.
The results are shown in Figure~\ref{swiss_roll}.
After that we computed pairwise Euclidean distances between the recovered points.
The resulting average relative error~\eqref{err} was equal to $Err = 0.4$.
In the second setup, we computed exact geodesic distances between points and applied 
the procedure with parameters $n = 2$ and $\varepsilon = 3$.
After that, we computed the Euclidean distances between the embedded points.
The resulting average relative error \eqref{err} was equal to $Err = 0.28$.

\begin{figure}[h]
	\noindent
	\includegraphics[width=\linewidth]{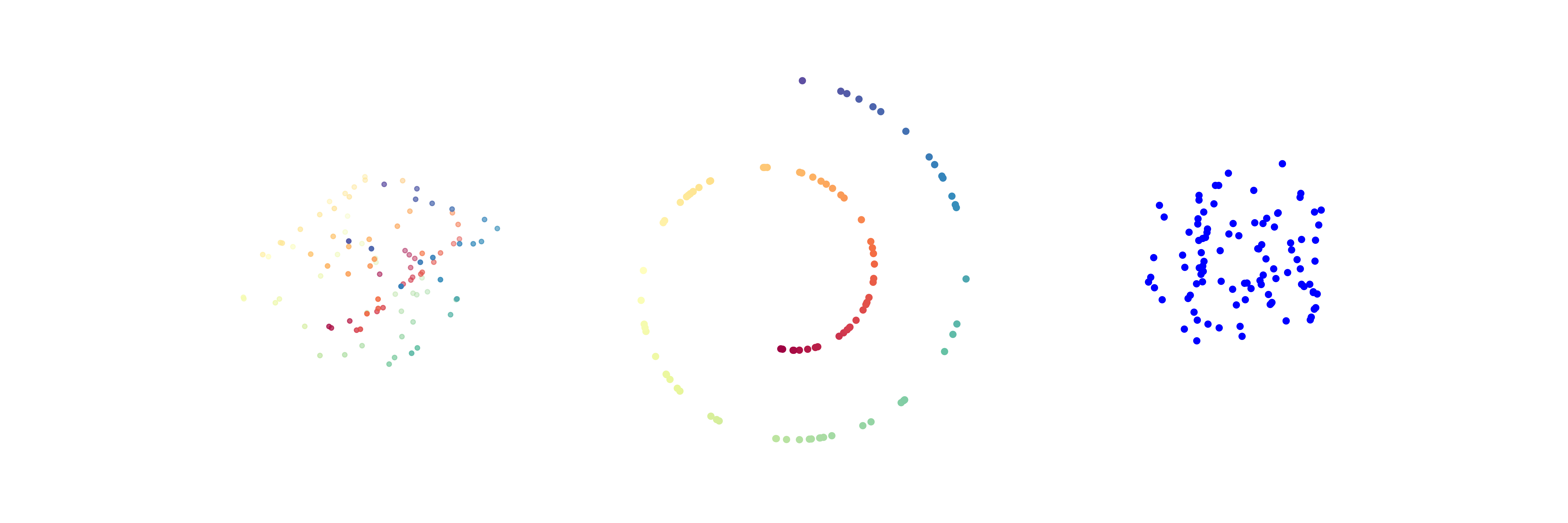}
	\includegraphics[width=\linewidth]{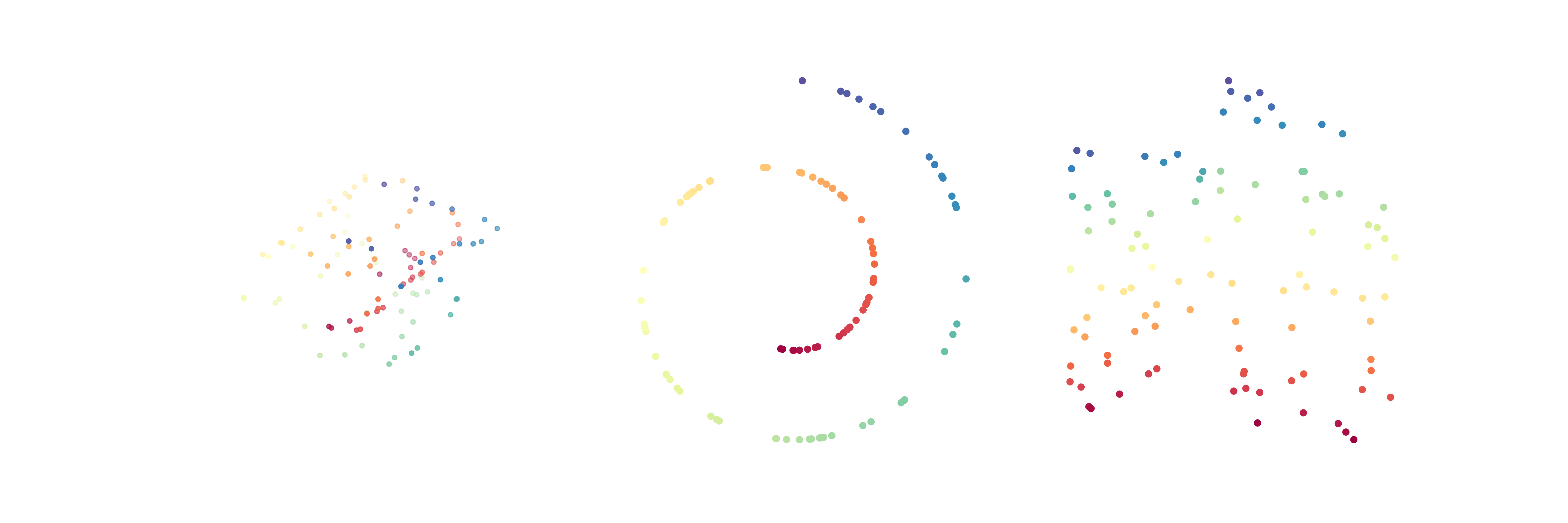}
	\caption{Reconstruction of the Swiss Roll from pairwise distances. Top line: sample 
		points in $\R^3$ (left), sample points, projection on the XOZ plane (center), 
		reconstructed Swiss Roll (right). Bottom line: sample points in $\R^3$ (left), sample 
		points, projection on the XOZ plane (center), embedding into $\R^2$ (right).}
	\label{swiss_roll}
\end{figure}

Finally, we provide the results of embedding Clifford torus into $\R^4$. Let us remind to a reader that Clifford torus is just the product of two circumferences $(1/\sqrt{2} S^1) \times (1/\sqrt{2} S^1)$, that is, the set of points
\[
	\left\{ \frac1{\sqrt{2}} (\cos \varphi, \sin \varphi, \cos \theta, \sin \theta) : 
	0 \leq \varphi < 2\pi, 0 \leq \theta < 2\pi \right\}. 
\]
We took $\sqrt{k} = 15$ equally spaced points on each circumference, so the total number of samples was equal to $k = 225$. After that, we applied the embedding procedure with parameters $d = 4$, $\varepsilon = 1.4$, and $C_2 = 0.2 / \pi$. The projections of the initial points and of the embedding into $\R^4$ are displayed in Figure \ref{torus}. The average relative distortion was equal to $Err = 0.21$. The distances between the embedded points were estimated in the same way as in Isomap.

\begin{figure}[h]
	\noindent\centering
	\includegraphics[width=\linewidth]{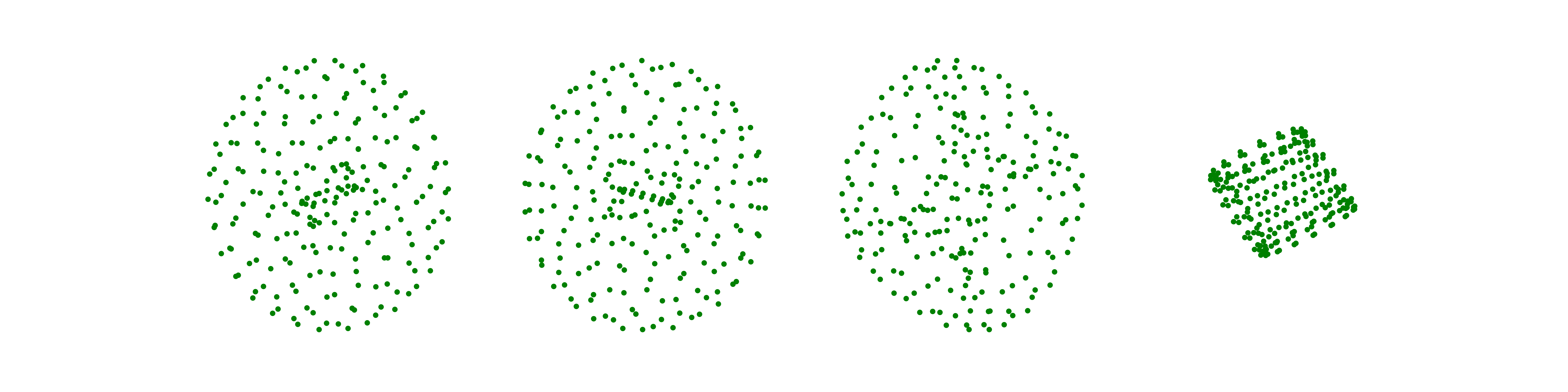}
	\includegraphics[width=\linewidth]{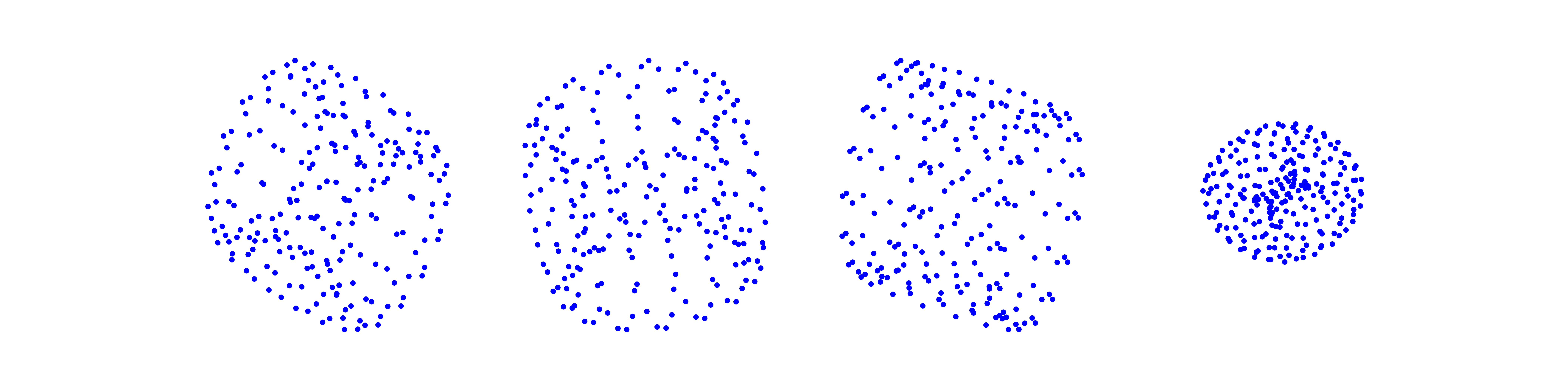}
	\caption{Reconstruction of Clifford torus from pairwise distances. Top line: projections of the original point cloud. Bottom line: projections of the points after embedding.}
	\label{torus}
\end{figure}

\appendix

\section{Auxiliary lemmata on sets of positive reach}

Throughout this section, $M\subset\R^n$ is a compact set, connected by rectifiable arcs,  
equipped with its intrinsic (geodesic) distance on $M$ denoted by $d_M$, and with $\Reach(M)=\alpha>0$.
In particular, this is true when $M$ is a $C^{1,1}$ smooth connected  compact Riemannian 
submanifold of $\R^n$.

\begin{lemma}
	\label{lm_locdistest1}
	Let $M$ be as above and let $u, v$ be any two points on $M$, such that 
	$d_M(u, v) < \alpha$.
	Then
	\[
	0\leq \dfrac{d_M(u,v)}{|u-v|}-1 \leq C_1 |u-v| 
	\]
	for all $\{u,v\}\subset M$ with $C_1 := \left(3\alpha\sqrt 3\right)^{-1}$.
\end{lemma}

\begin{proof}
	Since $|u-v|\leq d_M(u, v) < \alpha$, then
	by theorem~1 of~\cite{blw18}  we have 
	\[
	\frac{d_M(u, v)}{2\alpha} \leq \arcsin \frac{|u - v|}{2\alpha}.
	\]
	Let $f(t) := \arcsin t$, $t < 0.5$.
	Since $f(0) = 0$, $f'(0) = 1$ and
	\[
	f''(t) = \frac{x}{(1 - x^2)^{3/2}} \leq \frac4{3\sqrt3},
	\]
	we have
	\[
	f(t) \leq t + \frac{2t^2}{3\sqrt3}.
	\]
	Thus,
	\[
	\frac{d_M(u, v) - |u - v|}{2\alpha} \leq \frac{|u - 
		v|^2}{6\alpha^2\sqrt3},
	\]
	which yields
	\[
	\frac{d_M(u, v)}{|u - v|} - 1 \leq \frac{|u - v|}{3\alpha\sqrt3}
	\]
	as claimed.
\end{proof}

\begin{lemma}\label{lm_locdistest2}
	Let $M$ be as above. Then
	\[
	\dfrac{|u-v|}{d_M(u,v)}\geq \bar C_2 
	\]
	for all $\{u,v\}\subset M$ and for some $\bar C_2>0$. 
	In particular, one can take $\bar C_2:=(2/\pi)
	\wedge (2\alpha/D)$, where $D$ stands for the intrinsic diameter of $M$.
\end{lemma}

\begin{proof}
	If $|u-v|<2\alpha$, then
	by theorem~1 from~\cite{blw18}, one has
	\[
	d_M(u,v)\leq 2\alpha\arcsin\dfrac{|u-v|}{2\alpha},
	\]
	or, equivalently,
	\[
	\dfrac{|u-v|}{d_M(u,v)}\geq \dfrac{\sin t}{t},\quad\mbox{where } t:= \arcsin \dfrac{|u-v|}{2\alpha}.
	\]
	Since $0\leq t< \pi/2$, then
	\[
	\dfrac{|u-v|}{d_M(u,v)}\geq \dfrac{2}{\pi}.
	\]
	In the remaining case
	$|u-v|\geq 2\alpha$, one has
	\[
	\dfrac{|u-v|}{d_M(u,v)}\geq \dfrac{2\alpha}{D},
	\] 
	since $d_M(u,v)\leq D$.
\end{proof}

\bibliographystyle{plain}
\def\cprime{$'$} \def\cprime{$'$} \def\cprime{$'$} \def\cprime{$'$}

\end{document}